\documentclass[11pt]{amsart}
\usepackage{amsmath,amsthm,amssymb}
\usepackage[hidelinks]{hyperref}
\newcommand{\ZZ}{\mathbb Z}
\newcommand{\QQ}{\mathbb Q}
\newcommand{\CC}{\mathbb C}
\newcommand{\FF}{\mathbb F}
\newcommand{\HH}{\mathbb H}
\newcommand{\tr}{\operatorname{tr}}
\newcommand{\diag}{\operatorname{diag}}
\newcommand{\Span}{\operatorname{Span}}
\newcommand{\Aut}{\operatorname{Aut}}
\newcommand{\Sp}{\operatorname{Sp}}
\newcommand{\GL}{\operatorname{GL}}
\newcommand{\thc}[2]{\theta\!\begin{bmatrix}#1\\#2\end{bmatrix}}
\theoremstyle{plain}
\newtheorem{thm}{Theorem}[section]
\newtheorem{prop}[thm]{Proposition}
\newtheorem{cor}[thm]{Corollary}

\theoremstyle{definition}
\newtheorem{rem}[thm]{Remark}
\newtheorem{exa}[thm]{Example}

\begin{document}
\title{Theta Series with  Half-integral Characteristics }

\author{R. Salvati Manni}
\address{Universit\`a di Roma ``La Sapienza'', Dipartimento di Matematica,
Piazzale Aldo Moro 2, I-00185 Roma, Italy}
\email{salvati@mat.uniroma1.it}
\date{September 24, 2026}
\begin{abstract}
We study theta series with half-integral characteristics attached to odd
positive-definite unimodular lattices. In even rank these identities involve two rational lattices
of determinant one; in odd rank they involve coset unions, equivalently
theta series of a rational lattice with finite-order characters. We give
the transformation multipliers.  In even rank, we get non-vanishing theta series with odd characteristics.
 We emphasize ranks divisible by four,
where the rational lattices are   integral unimodular neighbors. In rank
congruent to four modulo eight, the signed even-characteristic module
gives a dimension dichotomy and a criterion for its smaller constituent.
In rank divisible by eight, we recover the permutation representations,
independence results, and global neighbor-span identities. Examples  are significative,
  they include the Schottky's form, the $\Delta$ form   and an 
 explicit rank-twenty lattice isometric to both of its odd neighbors.
\end{abstract}
\maketitle
\section{Introduction}
Let $L$ be an odd positive-definite unimodular lattice of rank $m\ge1$,
and let $w$ be a characteristic vector. Write $W=(w,w)$, so that
$W\equiv m\pmod8$, and let $E=L_e$ be the even sublattice.
The theta series $\thc{\varepsilon}{\delta}(L,\tau)$, indexed by binary
characteristics, transform into one another with explicit roots of unity.
Retaining these multipliers allows us to work in arbitrary rank.

The main character-sum identities use the normalization
\[
 F_{\varepsilon,\delta}
 =e\!\left(-\frac W4\varepsilon^t\delta\right)
   \thc{\varepsilon}{\delta}(L,\tau),\qquad e(z)=\exp(2\pi iz).
\]
Put $a=w/2$ and
\[
 K_0=E\cup(E+a),\qquad K_1=E\cup(L_o+a),\qquad L_o=L\setminus E.
\]
For every rank and degree we prove
\[
 \sum_{\varepsilon,\delta}(-1)^{t\varepsilon^t\delta}
 F_{\varepsilon,\delta}=2^g\theta^{(g)}(K_t),\qquad t=0,1.
\]
Here the theta series of a coset union means the sum over its $g$-fold
Cartesian product. When $m$ is even, $K_0,K_1$ are rational lattices of
determinant one. They are integral exactly when $4\mid m$. When $m$ is
odd, they are not additive; we instead express their theta series through
characters of the rational lattice $E^*$. In this case all odd-characteristic
series vanish identically.

The cases $4\mid m$ retain a particularly direct lattice interpretation.
For $m\equiv0\pmod8$ the two neighbors are even, and their characteristic
series form ordinary permutation modules of $\Sp(2g,\FF_2)$. For
$m\equiv4\pmod8$ the neighbors are odd and the action is signed. We prove
in this latter case that the dimension of the even-characteristic span is
either
\[
 a_g=\frac{(2^g+1)(2^{g-1}+1)}3
 \quad\text{or}\quad N_g^+=2^{g-1}(2^g+1).
\]
The smaller dimension occurs precisely when
$\theta^{(g)}(K_0)+\theta^{(g)}(K_1)=2\theta^{(g)}(L)$.
In particular it occurs when all three lattices are isometric.

We retain the permutation-module and global-span results in ranks divisible
by eight, stating their rank assumptions locally. Theta series of cosets
and double characteristics are studied in \cite{BS}; related lattice
identities occur in \cite{DBMS}. We use the finite representation results
of \cite{Fra,FSM}, the standard theta transformation formalism
\cite{F}, and Kneser-neighbor theory \cite{Kn,Ch,CS}.

\section*{Acknowledgments}
The author used ChatGPT-6 Astra Light  in discussions,  preparation, 
 deep proofreading, producing  some examples of this manuscript and  for elaborating  the cases $m\not\equiv 0\,\,  {\rm mod} \,4$
in section 2. The mathematical content was independently verified by the author, who takes responsibility for the manuscript.

\section{The even sublattice and rational extensions}
All quadratic spaces and lattices below are positive definite. A
\emph{rational lattice} is a free abelian group equipped with a
positive-definite rational bilinear form. Its determinant is the
determinant of a Gram matrix. The term \emph{unimodular} is reserved for
integral lattices $L=L^*$, where
\[
 L^*=\{x\in L\otimes\QQ:(x,L)\subseteq\ZZ\}.
\]
A characteristic vector $w\in L$ satisfies
$(w,x)\equiv(x,x)\pmod2$ for every $x\in L$.
Such vectors form one class modulo $2L$, and satisfy the standard
congruence \cite{Se}
\begin{equation}\label{charcongruence}
 W=(w,w)\equiv m\pmod8.
\end{equation}
Since $L$ is odd, $w\notin2L$.

Put $E=L_e$, and write $L_o=v_0+E$ with $v_0\in L_o$. Then
$[L:E]=2$, $\det E=4$, and $a=w/2\in E^*\setminus L$.
The four classes of $E^*/E$ are represented by $0,v_0,a,v_0+a$.
For every rank define the two sets
\begin{equation}\label{rationalsets}
 K_0=E\cup(E+a),\qquad K_1=E\cup(E+v_0+a).
\end{equation}

\begin{prop}\label{classification}
If $m$ is even, $E^*/E\cong(\ZZ/2\ZZ)^2$, and $K_0,K_1$ are
rational lattices of determinant one, with pairwise intersection $E$
with each other and with $L$. They are integral if and only if $4\mid m$.
If $m$ is odd, $E^*/E\cong\ZZ/4\ZZ$ is generated by $a+E$,
$L/E$ is its unique subgroup of order two, and neither $K_t$ is a lattice.
\end{prop}
\begin{proof}
We have $2a=w\in E$ precisely when $m$ is even, by
\eqref{charcongruence}. This proves the assertions about the discriminant
group and additivity. In even rank both extensions have index two over
$E$, hence determinant $4/2^2=1$. Since $(a,E)\subseteq\ZZ$,
integrality of $K_0$ is equivalent to $a^2=W/4\in\ZZ$. Moreover
\[
 (v_0+a)^2=W/4+(v_0,v_0)+(w,v_0)\equiv W/4\pmod{2\ZZ},
\]
so the same condition holds for $K_1$. In odd rank $2a+E=w+E=v_0+E$
is the nonzero class of $L/E$, so both remaining classes have order four.
\end{proof}

Thus the possibilities are
\[
\begin{array}{c|c|c}
 m\pmod8&E^*/E&K_0,K_1\\\hline
 0&(\ZZ/2\ZZ)^2&\text{even unimodular lattices}\\
 4&(\ZZ/2\ZZ)^2&\text{odd unimodular lattices}\\
 2,6&(\ZZ/2\ZZ)^2&\text{nonintegral rational lattices, determinant }1\\
 1,3,5,7&\ZZ/4\ZZ&\text{nonadditive unions of two cosets}
\end{array}
\]
For even $m$, their intersections have index two, so the neighbor
terminology extends to rational lattices; determinant one in this setting
does not assert integrality or self-duality.

\subsection{Integral neighbors when \texorpdfstring{$4\mid m$}{4 divides m}}
In these ranks we also use the customary notation
\begin{equation}\label{neighbors}
 L_1=K_0=E\cup(E+w/2),\qquad
 L_2=K_1=E\cup(L_o+w/2).
\end{equation}
The three lattices $L,L_1,L_2$ are the three integral unimodular
index-two extensions of $E$. For $m\equiv4\pmod8$ all three are odd
and have even sublattice $E$; for $8\mid m$ only $L$ is odd.
Changing $w$ to $w+2b$ preserves the ordering when $b\in E$ and
exchanges it when $b\in L_o$.

\begin{prop}\label{q1}
An odd unimodular lattice of rank divisible by eight has exactly two even
unimodular $2$-neighbors in its rational quadratic space, namely $L_1,L_2$.
\end{prop}
\begin{proof}
The intersection of an even neighbor with $L$ is an index-two sublattice
contained in $E$, hence equals $E$. Proposition~\ref{classification}
lists its three index-two extensions and their parities.
\end{proof}

\begin{exa}\label{ex}
For $L=I_m=\ZZ^m$ take $w=(1,\ldots,1)^t$. Then $E=D_m$, where
\[
 D_m=\{x\in\ZZ^m:\textstyle\sum_i x_i\equiv0\pmod2\}.
\]
For even $m$, $K_0=D_m\cup(D_m+\tfrac12(1,\ldots,1))$, and
changing the sign of the first coordinate exchanges $K_0,K_1$.
For $4\mid m$ this is the lattice $D_m^+$.
\end{exa}

\begin{prop}\label{p2}
For $\phi\in\Aut(L)$, write $\phi(w)=w+2v_w$ with $v_w\in L$.
The class of $v_w$ modulo $E$ is independent of the characteristic vector.
\end{prop}
\begin{proof}
Replacing $w$ by $w+2x$ replaces $v_w$ by $v_w+\phi(x)-x$.
The last difference has norm $2(x,x)-2(\phi(x),x)$, which is even.
\end{proof}

\begin{prop}\label{p4}
Let $8\mid m$. If an isometry $\psi:L_1\to L_2$ satisfies
$\psi(E)=E$, its linear extension preserves $L$ and sends $w$ to
$w+2v$ for some $v\in L_o$.
\end{prop}
\begin{proof}
On $E^*/E$ it fixes the unique odd-norm class and exchanges the other
two nonzero classes. Thus it preserves $L$, and exchanges $a+E$ and
$a+v_0+E$, which gives the assertion about $v$.
\end{proof}
For $m\equiv4\pmod8$, the three nonzero classes all have odd norm;
the same argument does not apply. In either case an automorphism of
$L$ exchanging the two specified neighbors is sufficient for the
vanishing criterion below. An abstract isometry $L_1\cong L_2$ does
not supply such an automorphism by this argument.

\section{Characteristics and transformation multipliers}
There is no rank restriction in this section. For $X=(x_1,\ldots,x_g)$
write $(X,X)=((x_i,x_j))_{ij}$, and put
\begin{equation}\label{definition}
 \thc{\varepsilon}{\delta}(L,\tau;w)
 =\sum_{\substack{x_j\in L+\varepsilon_jw/2\\1\le j\le g}}
 e\!\left(\frac12\tr((X,X)\tau)
                +\frac12\sum_j\delta_j(w,x_j)\right),
 \,.
\end{equation}
Here $\tau\in\HH_g$.
We first allow integral $\varepsilon,\delta$. Positive definiteness gives
normal convergence. For $u,v\in\ZZ^g$, reindexing gives
\begin{equation}\label{reduction}
 \thc{\varepsilon+2u}{\delta}=\thc{\varepsilon}{\delta},\qquad
 \thc{\varepsilon}{\delta+2v}
   =(-1)^{W\varepsilon^tv}\thc{\varepsilon}{\delta}.
\end{equation}
In odd rank this sign must be retained. Unless stated otherwise we use
the fixed representatives $\varepsilon,\delta\in\{0,1\}^g$.
A characteristic is even or odd according to $\varepsilon^t\delta\pmod2$.

\begin{prop}\label{modularity}
For $T=T^t\in\ZZ^{g,g}$, $d=\diag T$, and $A\in\GL(g,\ZZ)$,
the transformation formulas are
\begin{align}
 \thc{\varepsilon}{\delta}(L,\tau+T)
 &=e\!\left(-\frac W8\varepsilon^tT\varepsilon
             -\frac W4\varepsilon^td\right)
   \thc{\varepsilon}{\delta+T\varepsilon+d}(L,\tau),\label{translation}\\
 \thc{\varepsilon}{\delta}(L,A\tau A^t)
 &=\thc{A^t\varepsilon}{A^{-1}\delta}(L,\tau),\label{linear}\\
 \thc{\varepsilon}{\delta}(L,-\tau^{-1})
 &=\det(\tau/i)^{m/2}e\!\left(\frac W4\varepsilon^t\delta\right)
   \thc{\delta}{-\varepsilon}(L,\tau).\label{inversion}
\end{align}
The determinant power in \eqref{inversion} is the holomorphic branch
equal to one at $\tau=iI_g$. Characteristics on the right are integral
and are reduced using \eqref{reduction}.
\end{prop}
\begin{proof}
Write $x_j=v_j+\varepsilon_jw/2$. In a translation, the terms depending
on $V=(v_j)$ in the extra exponent are, modulo integers,
$\frac12\sum_jd_j(w,v_j)+\frac12\sum_j(T\varepsilon)_j(w,v_j)$.
The remaining quadratic constant is $W\varepsilon^tT\varepsilon/8$.
Comparing with the new linear character gives \eqref{translation}.
Changing variables $X\mapsto XA$ gives \eqref{linear}.
Poisson summation, with $L^*=L$, gives \eqref{inversion}: the Fourier
transform interchanges the shift and linear character, and its Gaussian
factor is the stated determinant power. See also \cite{F}.
\end{proof}

Put $\Gamma_g=\Sp(2g,\ZZ)$. We use the affine action
\begin{equation}\label{action}
 M\circ\begin{bmatrix}\varepsilon\\\delta\end{bmatrix}
 =\begin{bmatrix}D&-C\\-B&A\end{bmatrix}
  \begin{bmatrix}\varepsilon\\\delta\end{bmatrix}
 +\begin{bmatrix}\diag(CD^t)\\\diag(AB^t)\end{bmatrix}\pmod2,
 \quad M=\begin{pmatrix}A&B\\C&D\end{pmatrix}.
\end{equation}
For even $m$, let $k=m/2$ and use
$f|_kM=\det(C\tau+D)^{-k}f(M\tau)$. In odd rank use the metaplectic
cover: for a lift $\widetilde M=(M,\varphi)$ with
$\varphi(\tau)^2=\det(C\tau+D)$, put
$f|_{m/2}\widetilde M=\varphi^{-m}f(M\tau)$.
The generator formulas give, in both cases,
\begin{equation}\label{transformation}
 \theta[c]|_{m/2}\widetilde M
   =\chi_{m,w}(\widetilde M,c)\,
     \theta[M^{-1}\circ c],
\end{equation}
with a root of unity $\chi_{m,w}$; in even rank the lift is unnecessary.
Equations \eqref{translation}--\eqref{inversion} and \eqref{reduction}
specify the multipliers, including their representative conventions.
They depend on $W\pmod8$, hence on $m\pmod8$, and on the lift in
half-integral weight. The diagonal reweighting used below conjugates
these monomial transformation matrices.
For a characteristic $c$ whose series is nonzero, restrict to its affine
stabilizer in $\Gamma_g$ (or its metaplectic inverse image). The slash
composition law makes $\chi_{m,w}(\,\cdot\,,c)$ a finite-order character
of this stabilizer. Thus each nonzero component is a scalar theta form
of weight $m/2$ with that character; the full family has the monomial
transformation law \eqref{transformation}.

The affine action preserves parity and is transitive on each parity
class. For $8\mid m$ every multiplier under the integral-weight slash
action is one. For $m\equiv4\pmod8$ it is a sign: translation contributes
$(-1)^{\varepsilon^tT\varepsilon}$, a linear substitution contributes
one, and full inversion contributes $(-1)^g$.
For the other congruence classes the phases in the displayed formulas
cannot in general be suppressed.

\begin{rem}\label{representative}
For $w'=w+2b$ one has
\[
 \thc{\varepsilon}{\delta}(L,\tau;w')
 =(-1)^{(b,w)\varepsilon^t\delta}
   \thc{\varepsilon}{\delta}(L,\tau;w).
\]
Indeed, the shifted summation sets are unchanged, and the change in the
linear character has that constant value on each such set.
Thus the even series are independent of this choice and the odd series
change by a common sign.
\end{rem}

\begin{cor}\label{s4}
Every even-characteristic series is nonzero. In a fixed degree the odd
series are either all zero or all nonzero. If  $L$ has
  odd  rank they are all zero.
\end{cor}
\begin{proof}
Transitivity and the nonzero multipliers prove the first two assertions,
since $\theta[0,0]$ has constant coefficient one. Replacing every $x_j$
by $-x_j$ in \eqref{definition} gives
$\theta[\varepsilon,\delta]=(-1)^{W\varepsilon^t\delta}
\theta[\varepsilon,\delta]$, proving the last assertion.
\end{proof}
 Similarly 
 \begin{cor}\label{s4}
 $$L= L_1\perp L_2\perp\dots\perp L_k$$  with  at least  a $L_i$ of odd  rank , the odd
series  are all zero.
\end{cor}
\
\begin{proof} We apply the fact that
 \(L=L_1\perp L_2\) and \(w=w_1+w_2\), where \(w_i\) is characteristic for \(L_i\) gives
\[
\theta\!\begin{bmatrix}\varepsilon\\ \delta\end{bmatrix}(L,\tau;w)
=
\theta\!\begin{bmatrix}\varepsilon\\ \delta\end{bmatrix}(L_1,\tau;w_1)\,
\theta\!\begin{bmatrix}\varepsilon\\ \delta\end{bmatrix}(L_2,\tau;w_2).
\]
\end{proof}

\begin{prop}\label{p5}
If $\phi\in\Aut(L)$ satisfies $\phi(w)=w+2v$ with $v\in L_o$, all
odd-characteristic series vanish in every degree. This applies, in
particular, whenever $L$ represents one.
\end{prop}
\begin{proof}
Changing variables by $\phi$ and applying Remark~\ref{representative}
multiplies an odd-characteristic series by $-1$. If $L$ has a norm-one
vector, it splits as $\ZZ\perp L'$; choose $w$ with first coordinate one
and change the sign of this coordinate.
\end{proof}

\section{Character sums and theta traces in arbitrary rank}
If $S$ is a lattice or a finite union of lattice cosets, write
\[
 \theta^{(g)}(S,\tau)=\sum_{X\in S^g}
 e\!\left(\tfrac12\tr((X,X)\tau)\right).
\]
This notation does not assert that $S$ is additive. With binary
representatives define
\begin{equation}\label{normalized}
 F_{\varepsilon,\delta}
 =e\!\left(-\tfrac W4\varepsilon^t\delta\right)
   \thc{\varepsilon}{\delta}(L,\tau),\qquad
 \mathcal T_{L,g}^{\pm}
 =\sum_{(-1)^{\varepsilon^t\delta}=\pm1}F_{\varepsilon,\delta}.
\end{equation}
In terms of $x_j=v_j+\varepsilon_ja$,
\[
 F_{\varepsilon,\delta}
 =\sum_{v_j\in L}(-1)^{\sum_j\delta_j(w,v_j)}
 e\!\left(\tfrac12\tr((V+a\varepsilon^t,V+a\varepsilon^t)\tau)\right).
\]
Thus the normalization extracts precisely a character of $L/E$.

\begin{thm}\label{neighbor}
For every rank $m\ge1$ and degree $g\ge1$,
\begin{align}
 \sum_{\varepsilon,\delta}(-1)^{t\varepsilon^t\delta}
 F_{\varepsilon,\delta}&=2^g\theta^{(g)}(K_t),\quad t=0,1,
 \label{generaltrace}\\
 \mathcal T_{L,g}^{\pm}
 &=2^{g-1}\bigl(\theta^{(g)}(K_0)\pm\theta^{(g)}(K_1)\bigr).
 \label{weightedtrace}
\end{align}
\end{thm}
\begin{proof}
For fixed $\varepsilon$, summation over $\delta$ gives the product
\[
 \prod_{j=1}^g\sum_{\delta_j=0}^1
 (-1)^{\delta_j((w,v_j)+t\varepsilon_j)}.
\]
This is $2^g$ precisely when $v_j\in E$ for $t\varepsilon_j=0$
and $v_j\in L_o$ for $t\varepsilon_j=1$; otherwise it is zero.
Summing over $\varepsilon$ leaves exactly $K_t^g$. Adding and subtracting
the two identities gives \eqref{weightedtrace}.
\end{proof}

For comparison define the unweighted sums
\[
 \theta_{L,g}^{\pm}=\sum_{(-1)^{\varepsilon^t\delta}=\pm1}
 \thc{\varepsilon}{\delta}(L,\tau).
\]
If $m$ is even, then
\begin{equation}\label{evenranktraces}
\begin{aligned}
 \theta_{L,g}^+&=2^{g-1}(\theta^{(g)}(K_0)+\theta^{(g)}(K_1)),\\
 \theta_{L,g}^-&=(-1)^{m/2}2^{g-1}
 (\theta^{(g)}(K_0)-\theta^{(g)}(K_1)).
\end{aligned}
\end{equation}
Indeed $F_{\varepsilon,\delta}=(-1)^{(W/2)\varepsilon^t\delta}
\theta[\varepsilon,\delta]$. This sign is constant on each parity class.

\subsection{The principal specialization: \texorpdfstring{$4\mid m$}{4 divides m}}
Now $F_{\varepsilon,\delta}=\theta[\varepsilon,\delta]$, and
\eqref{weightedtrace} becomes the unweighted integral-neighbor formula
\begin{equation}\label{trace}
 \boxed{\quad\theta_{L,g}^{\pm}
 =2^{g-1}\bigl(\theta^{(g)}(L_1)\pm\theta^{(g)}(L_2)\bigr).\quad}
\end{equation}
Both signs have the same meaning for $m\equiv0$ and $m\equiv4\pmod8$;
only the parity of the neighbors and the transformation multipliers change.

\begin{cor}\label{c1}
If $8\mid m$ and an isometry $L_1\to L_2$ preserves $E$, every
odd-characteristic series of $L$ vanishes in every degree.
\end{cor}
\begin{proof}
Apply Propositions~\ref{p4} and \ref{p5}.
\end{proof}
Equality of the two neighbor theta series in a given degree proves the
vanishing of the odd trace; without an additional argument it does not
prove the vanishing of every summand.

\subsection{Odd rank and characters of a rational quadratic form}
Suppose now that $m$ is odd. We may take $v_0=w$, so that
$K_0=E\cup(E+a)$ and $K_1=E\cup(E+3a)$. Put
\[
 R=E^*=E+\ZZ a,\qquad \det R=\tfrac14,
 \qquad\eta(x)=i^j\quad(x\in ja+E).
\]
This character has order four. For
$r=(r_1,\ldots,r_g)\in\{0,1,2,3\}^g$, define
\[
 \Xi_g(R;r;\tau)=\sum_{X\in R^g}
 \prod_j\eta(x_j)^{r_j}
 e\!\left(\tfrac12\tr((X,X)\tau)\right).
\]
The elementary indicator identities on $R/E$ give
\begin{equation}\label{fourierodd}
 2^g\theta^{(g)}(K_t)
 =2^{-g}\sum_{r\in\{0,1,2,3\}^g}
 \prod_{j=1}^g(1+i^{(-1)^{t+1}r_j})\,\Xi_g(R;r;\tau).
\end{equation}
For example $\mathbf1_{K_0}(x)=\frac14\sum_{r=0}^3
(1+i^{-r})\eta(x)^r$, and the formula for $K_1$ has $i^r$ in place
of $i^{-r}$. Multiplying these identities column by column proves
\eqref{fourierodd}. Choosing a basis of $R$ realizes each $\Xi_g$ as
a theta series of a positive-definite rational Gram matrix with an
additive character. Thus the arbitrary-rank formulas do have a rational
quadratic-form interpretation even when two neighbor lattices do not exist.

Negation sends $K_0$ to $K_1$, so their theta series agree in every degree,
in agreement with $\mathcal T_{L,g}^-=0$. In odd rank the factor in
\eqref{normalized} need not be constant on the even characteristics;
the weighted and unweighted even traces must not be identified.

\section{Signed representations when \texorpdfstring{$m\equiv4\pmod8$}{m is 4 modulo 8}}
Let $G_g=\Sp(2g,\FF_2)$ and let $\mathcal C_g^{\pm}$ denote the
even and odd characteristics, of cardinalities
$N_g^{\pm}=2^{g-1}(2^g\pm1)$. Write
\[
 \Theta_{L,g}^{\pm}=\Span_{\CC}\{\theta[c](L):c\in\mathcal C_g^{\pm}\}.
\]
The signs on even characteristics in this rank are exactly those of the
fourth powers of classical theta constants. They factor through $G_g$.
Let $V_g^{\mathrm{sgn},+}$ be the resulting signed module with basis
$e_c$ for $c\in\mathcal C_g^+$. With the left action on functions
$M\cdot f=f|_{m/2}M^{-1}$, the map
\[
 \rho:V_g^{\mathrm{sgn},+}\longrightarrow\Theta_{L,g}^+,
 \qquad e_c\longmapsto\theta[c](L)
\]
is an equivariant surjection. The module is induced from the nontrivial
sign character of the stabilizer of the zero characteristic.

For $c=(\varepsilon,\delta)$ and $d=(\alpha,\beta)$ let
$\langle c,d\rangle=\varepsilon^t\beta+\delta^t\alpha\pmod2$.
The commuting operator
\[
 H_g e_c=\sum_{d\in\mathcal C_g^+}(-1)^{\langle c,d\rangle}e_d
\]
has two irreducible eigenspaces, with eigenvalues $2^g$ and $-2^{g-1}$,
and respective dimensions
\begin{equation}\label{signeddimensions}
 a_g=\frac{(2^g+1)(2^{g-1}+1)}3,\qquad
 b_g=\frac{2^{2g}-1}3.
\end{equation}
These facts are established in \cite[Sections 2--4]{FSM}, following
Frame and Fay; they identify the signed representation used here.

\begin{thm}\label{rankfourdim}
For every odd unimodular $L$ with $m\equiv4\pmod8$,
\[
 \dim\Theta_{L,g}^+\in\{a_g,N_g^+\}.
\]
The dimension equals $a_g$ if and only if
\begin{equation}\label{smallercriterion}
 \theta^{(g)}(L_1)+\theta^{(g)}(L_2)=2\theta^{(g)}(L).
\end{equation}
In particular, $L\cong L_1\cong L_2$ implies dimension $a_g$ in every degree.
\end{thm}
\begin{proof}
Put $z=\sum_{\delta\in\FF_2^g}e_{(0,\delta)}$.
Character orthogonality gives $H_gz=2^gz$.
The constant coefficient of $\rho(z)$ is $2^g$, since each summand
has zero shift and constant coefficient one. Thus the $2^g$ eigenspace
has nonzero image and, by irreducibility, survives injectively.
The kernel is therefore either zero or the $-2^{g-1}$ eigenspace.

The nonzero vector
$u=(H_g-2^g)e_0=\sum_{c\in\mathcal C_g^+}e_c-2^ge_0$
lies in that latter eigenspace and generates it as a $G_g$-module.
Consequently the smaller dimension occurs precisely when $\rho(u)=0$.
By \eqref{trace}, this condition is
$2^{g-1}(\theta^{(g)}(L_1)+\theta^{(g)}(L_2))-2^g\theta^{(g)}(L)=0$.
\end{proof}
In particular, the value $b_g$ alone cannot occur: it would discard the
constituent detected by the constant coefficient. The representation is
signed, so the arguments involving a trivial invariant line in the next
section do not apply here.

\section{Permutation representations when \texorpdfstring{$8\mid m$}{8 divides m}}
Assume throughout this section that $m\equiv0\pmod8$, and put
$G_g=\Gamma_g/\Gamma_g(2)=\Sp(2g,\FF_2)$.
Let $\mathcal C_g^+$ and $\mathcal C_g^-$ be the sets of even and odd
characteristics. Their cardinalities are
\[
 N_g^+=2^{g-1}(2^g+1),\qquad N_g^-=2^{g-1}(2^g-1).
\]
Let $V_g^{\pm}=\CC[\mathcal C_g^{\pm}]$ be the permutation modules and
\[
 \Theta_{L,g}^{\pm}
 =\Span_{\CC}\{\theta[c](L,\tau):c\in\mathcal C_g^{\pm}\}.
\]
To use a left action on functions, let $M$ act by $f\mapsto f|_kM^{-1}$.
Proposition~\ref{modularity} gives equivariant surjections
\[
 \rho_g^{\pm}:V_g^{\pm}\longrightarrow\Theta_{L,g}^{\pm},
 \qquad e_c\longmapsto\theta[c](L,\tau).
\]
The stabilizer of the zero characteristic is the theta subgroup modulo
$\Gamma_g(2)$; thus $V_g^+$ is also its induced trivial representation.

We recall the decompositions \cite{Fra}
\begin{equation}\label{decomposition}
 V_g^+=\mathbf1\oplus A(g),\qquad
 V_g^-=\mathbf1\oplus B(g),
\end{equation}
where $A(g)$ is irreducible of dimension $N_g^+-1$, and $B(g)$ is
irreducible of dimension $N_g^--1$ for $g\ge2$. We set $B(1)=0$.
One way to see the irreducibility is to identify characteristics with
quadratic refinements of the symplectic form on $\FF_2^{2g}$.
For a fixed refinement $q$, the other refinements of the same Arf invariant
are $q+\langle v,\cdot\rangle$ with $q(v)=0$. Its orthogonal group is
transitive on the nonzero such vectors. Hence each nontrivial permutation
action is doubly transitive, and its augmentation module is irreducible.

\begin{thm}\label{distinguish}
The map $\rho_g^+$ is an isomorphism. In particular, all the
$N_g^+$ even-characteristic theta series are linearly independent.
\end{thm}
\begin{proof}
The invariant vector $\sum_{c\in\mathcal C_g^+}e_c$ maps to
\[
 \theta_{L,g}^+=2^{g-1}\bigl(\theta^{(g)}(L_1)+\theta^{(g)}(L_2)\bigr),
\]
whose constant coefficient is $2^g$.
Choose $v\in L$ of odd norm $n$. The coefficient of $\theta[0,0](L)$
at the Gram index $\diag(n,0,\ldots,0)$ is positive, whereas that of
$\theta_{L,g}^+$ is zero, since its two neighbors are even.
Thus the image is not contained in the invariant line.
The map is nonzero on both irreducible constituents in
\eqref{decomposition}, and is therefore injective as well as surjective.
\end{proof}

Set $D_g=\theta^{(g)}(L_1)-\theta^{(g)}(L_2)$.
The image of the invariant vector in $V_g^-$ is $2^{g-1}D_g$.
It follows that $D_g\ne0$ is necessary for $\rho_g^-$ to be an
isomorphism. For $g=1$ this condition is also sufficient. We do not prove its
sufficiency in the same degree for $g\ge2$.

For $h\ge2$, the Siegel operator is
\[
 \Phi(f)(\tau)=\lim_{t\to+\infty}
 f\!\begin{pmatrix}\tau&0\\0&it\end{pmatrix}.
\]
If $\varepsilon=(\bar\varepsilon,\varepsilon_h)^t$ and
$\delta=(\bar\delta,\delta_h)^t$, the defining series gives
\begin{equation}\label{phi}
 \Phi\!\left(\thc{\varepsilon}{\delta}^{(h)}(L)\right)
 =\begin{cases}
 \thc{\bar\varepsilon}{\bar\delta}^{(h-1)}(L),&\varepsilon_h=0,\\
 0,&\varepsilon_h=1.
 \end{cases}
\end{equation}
Indeed, the limit retains only a zero last column, which is impossible
when $\varepsilon_h=1$, since $w/2\notin L$.
Each fixed odd characteristic of degree $h-1$ has two odd extensions with
$\varepsilon_h=0$, according to the choice $\delta_h=0,1$. Hence
\begin{equation}\label{phitrace}
 \Phi(\theta_{L,h}^-)=2\theta_{L,h-1}^-,\qquad
 \Phi(D_h)=D_{h-1}.
\end{equation}
These identities do not require any nonvanishing assumption.

\begin{prop}\label{odd-distinguish}
If $D_g\not\equiv0$ for some $g\ge1$, then
$\rho_{g+1}^-:V_{g+1}^-\to\Theta_{L,g+1}^-$ is an isomorphism.
\end{prop}
\begin{proof}
Since $\Phi(D_{g+1})=D_g\ne0$, the invariant constituent has nonzero
image in degree $g+1$. The odd traces in both genera are nonzero, so
Corollary~\ref{s4} shows that every odd-characteristic series in these
genera is nonzero.
In degree $g+1$, choose an odd characteristic with
$\varepsilon_{g+1}=0$ and another with $\varepsilon_{g+1}=1$.
Their nonzero theta series have, respectively, nonzero and zero images
under $\Phi$, so they are linearly independent.
Thus the image is not just the invariant line. The irreducibility of
$B(g+1)$ now proves the assertion.
\end{proof}

\begin{cor}\label{dimensions}
For $g\ge2$, the representation-theoretic possibilities are
\[
 \dim\Theta_{L,g}^-\in\{0,1,N_g^--1,N_g^-\}.
\]
For $g=1$, the dimension is zero or one.
\end{cor}
\begin{proof}
The space is a quotient of the multiplicity-free module in
\eqref{decomposition}. Its dimension is therefore the sum of the
dimensions of a subset of its constituents.
\end{proof}
This does not assert that every possibility occurs.

\section{Vanishing and nonvanishing across genera}
The Siegel-operator calculation \eqref{phi} is valid in every rank,
and also holds for $F_{\varepsilon,\delta}$. In particular,
\[
 \Phi(\mathcal T_{L,h}^{\pm})=2\mathcal T_{L,h-1}^{\pm}.
\]
For even $m$ put $\mathcal D_g=\theta^{(g)}(K_0)-\theta^{(g)}(K_1)$.
Then $\Phi(\mathcal D_h)=\mathcal D_{h-1}$.

\begin{prop}\label{s3}
Let $m$ be even. If $\mathcal D_g\ne0$, every odd-characteristic series
in degree $g$ is nonzero. If $\mathcal D_r\ne0$, this remains true in
every degree $h\ge r$.
\end{prop}
\begin{proof}
Equation \eqref{evenranktraces} gives a nonzero odd trace, and
Corollary~\ref{s4} gives nonvanishing of every summand. Apply the
Siegel operator repeatedly for the assertion in larger genera.
\end{proof}

\begin{cor}\label{t3}
If $8\mid m$ and the neighbor difference is nonzero in degree $r$,
the odd-characteristic series are linearly independent in every degree
$h\ge r+1$. If $r=1$, they are also independent in degree one.
\end{cor}
\begin{proof}
Use Proposition~\ref{odd-distinguish} and the preceding proposition.
In degree one there is just one odd characteristic.
\end{proof}
In odd rank the odd-characteristic series vanish in all genera, so there
is no analogous nonvanishing assertion for them.

\section{Examples}

\subsection{The classical lattice and the smallest ranks}
For $I_m$ with its standard characteristic vector,
$\theta[c](I_m)=\vartheta[c]^m$, where $\vartheta$ is the classical
theta constant. The odd series vanish. The even powers are linearly
independent for every positive $m\ne4$, whereas for $m=4$ their span has
dimension $a_g$; see \cite{FSM}.

For $m=1$, $E=2\ZZ$, $a=1/2$, and
\[
 K_0=2\ZZ\cup(2\ZZ+1/2),\qquad
 K_1=2\ZZ\cup(2\ZZ+3/2).
\]
They are negatives of one another but neither is a lattice; $R=\frac12\ZZ$.
For $m=2$, the same construction on $I_2$ gives rational determinant-one
lattices with a vector $a=(1/2,1/2)$ of norm $1/2$.
These examples display the distinctions required outside $4\mid m$.

For $m=4$, the two integral neighbors of $I_4$ are both isometric to
$I_4$, and Theorem~\ref{rankfourdim} recovers dimension $a_g$.
For $4\mid m$ and $m>4$, the neighbors of $I_m$ are isometric to
$D_m^+$, whose minimum is $\min(2,m/4)>1$; hence they are not isometric
to $I_m$.

\subsection{Rank twelve}
Let $L=D_{12}^+$, with $s=\tfrac12(1,\ldots,1)$, so that
$L=D_{12}\cup(D_{12}+s)$ and $L_e=D_{12}$.
The vector $w=2e_1$ belongs to $L$ and is characteristic: its pairing
is even on $D_{12}$ and odd on $D_{12}+s$, as are the respective norms.
Thus \eqref{neighbors} gives
\[
 L_1=D_{12}\cup(D_{12}+e_1)=\ZZ^{12},\qquad
 L_2=D_{12}\cup(D_{12}+s+e_1)\cong D_{12}^+.
\]
The last isometry changes the sign of the first coordinate.
The first neighbor has 24 vectors of norm one, whereas the second has
none. With $q=e(\tau)$, the odd degree-one trace is therefore
\[
 \theta_{L,1}^-=\theta^{(1)}(\ZZ^{12})-\theta^{(1)}(D_{12}^+)
 =24q^{1/2}+O(q).
\]
Proposition~\ref{s3} proves nonvanishing in every degree.
This is a rank congruent to four modulo eight, so the two neighbors
are odd; the permutation-module conclusions for ranks divisible by eight do not apply.

\subsection{Rank sixteen}
There is a unique odd unimodular lattice of rank sixteen without vectors
of norm one, usually denoted $(D_8\oplus D_8)^+$; its even neighbors
are $E_8\oplus E_8$ and $D_{16}^+$, see \cite{CS,DBMS}.
Their theta series agree in genera one, two, and three, and differ in
degree four \cite{BFW}. Thus every odd-characteristic series is nonzero
for $g\ge4$, and, according to Proposition \ref{odd-distinguish} these series are linearly independent for $g\ge5$.
Really, using standard facts of  representation   of finite group and the  glueing method, cf.\cite{OPSY},  any  theta series of odd characteristic of $(D_8\oplus D_8)^+$ can be related 
 to a polynomial in the even theta series related to the standard lattice $I_8$,  in fact
 with the characteristic vector \(w=(2e_1,2e_1)\) in the two \(D_8\) blocks, we have
\[
\theta[c](L,\tau;w)
=
2^{-2g}
\sum_{p,q\ {\rm even}}
(-1)^{\langle p,q\rangle+\langle c,p+q\rangle}
\vartheta_p(\tau)^8\vartheta_q(\tau)^8 .
\]
  This  fact implies that the map  $\rho^-$ is injective

In lower degree  they vanish.  In fact , in degree three, we have
\[
V_3^{-}=\mathbf1\oplus\mathbf{27}_a.
\]

 Now,   $\mathbf{27}_a$ does not occur in the  decomposition of the space  of modular forms of weight 8  relative to the group $\Gamma_3(2)$ as established in \cite{DPvG}.   Both constituents therefore map to zero. The Siegel operator then gives individual vanishing in degree two.
In degree one the single odd series vanishes. 

\subsection{Rank twenty: three isometric odd extensions}
We give an explicit construction of the lattice with root system $D_4^5$.
Identify $D_4^*/D_4$ with the additive group of
$\FF_4=\{0,1,\omega,\omega^2\}$, where $\omega^2+\omega+1=0$.
Every nonzero class has minimum norm one, and the discriminant quadratic
form is one modulo $2\ZZ$ on each nonzero class. Triality realizes
multiplication by $\omega$ on this group.

In $\FF_4^5$ take the linear code
\[
 C=\Span_{\FF_4}\{(1,0,1,\omega,\omega^2),
                         (0,1,1,\omega^2,\omega)\}.
\]
It is Hermitian self-orthogonal. Every nonzero word has weight four:
up to scalar its five possibilities come from the second generator or
the first generator plus $t$ times the second, $t\in\FF_4$.
Let $E$ be the inverse image of $C$ in $(D_4^*)^5$. It is even integral,
has determinant $4^5/16^2=4$, and its norm-two vectors are precisely
the roots of $D_4^5$.

Put $c=(1,1,1,0,0)$. It lies in the Hermitian orthogonal code
$C^{\perp_H}$ but not in $C$. Since $C^{\perp_H}/C$ has dimension one,
its three nonzero classes are $c,\omega c,\omega^2c$.
The inverse images
\[
 M_t\longleftrightarrow C+\FF_2tc,\qquad t\in\FF_4^*,
\]
are the three integral unimodular extensions of $E$. Every word in
$C+tc$ has odd weight: explicitly there are ten words of weight three
and six of weight five. Thus each $M_t$ is odd, has minimum two and
root system $D_4^5$, and has even sublattice $E$.
Simultaneous triality in the five factors preserves $C$ and cyclically
permutes these three extensions. Taking any one as $L$ therefore gives
$L\cong L_1\cong L_2$ and, by Theorem~\ref{rankfourdim},
\[
 \dim\Theta_{L,g}^+=a_g\qquad(g\ge1).
\]
This example has no norm-one vector and hence no orthogonal $I_4$ summand.

There is also an automorphism fixing $M_1$ and exchanging the other
two extensions: conjugate all five $\FF_4$ coordinates and interchange
the fourth and fifth coordinates. These operations lift to isometries
of $D_4^5$; the composite fixes both code generators and $c$, and
exchanges $\omega c$ with $\omega^2c$. Proposition~\ref{p5} therefore
shows that all odd-characteristic series of this lattice vanish.
\begin{rem} This construction  generalises at higher rank $m\equiv 4, \, {\rm mod}\,8$ in two different way
 There are  odd unimodular lattices $C$ with root system $D_4^{N-1}$, $N\geq 6$.  
 We can  also take  lattices $L= M\perp C$  with $M$ unimodular  even.
 In all these cases we have    triality.
\end{rem}

\subsection{Rank twenty-four}
Borcherds classified the 156 odd unimodular lattices of rank twenty-four
without norm-one vectors and recorded their even neighbors \cite{Bor}.
There are exactly 34 of these lattices with isometric neighbors; see
Figure~17.1 in \cite{CS}.

 For each of these 34 lattices, all odd-characteristic theta series vanish identically in every degree. Indeed, 
Borcherds’s classification shows that an automorphism of the odd lattice exchanges its two isometric 
even neighbours, and Proposition \ref{p5} applies.

The odd Leech lattice has neighbors the Leech lattice and the Niemeier
lattice with root system $A_1^{24}$. For this ordering, the degree-one
 odd trace is $-48\Delta$, where $\Delta=q+O(q^2)$.
Hence its odd-characteristic series are nonzero and linearly independent
in every degree, by Corollary~\ref{t3}.

More generally, for a Niemeier lattice $M$ with Coxeter number $h_M$
(and $h_M=0$ for the Leech lattice), one has \cite{NT}
\[
 \theta^{(1)}(M)=E_4^3+(24h_M-720)\Delta.
\]
Thus an odd lattice with even neighbors of different Coxeter numbers
has nonzero odd-characteristic series in every degree.
The five pairs of distinct Niemeier root systems with equal Coxeter
numbers are
\[
\begin{array}{c|c|c}
 h&\text{first root system}&\text{second root system}\\\hline
 30&E_8^3&D_{16}E_8\\
 18&A_{17}E_7&D_{10}E_7^2\\
 12&A_{11}D_7E_6&E_6^4\\
 10&A_9^2D_6&D_6^4\\
 6&A_5^4D_4&D_4^6
\end{array}
\]
Their theta series agree through degree three, by \cite{NT}.
Whenever the two neighbors of an odd lattice form one of these pairs,
its odd trace therefore vanishes in these genera. The individual  cases can be checked  as in the rank 16  case in degree two
Thus  also in these cases  we get that i the individual form  vanish. 

In degree three, when  $h=30$ the odd lattice is $E_8\perp(D_8\oplus D_8)^+$, the vanishing is induced by the discussion done  for rank 24.
 We did not compute  the other cases. In  any case, we can  say that  we  do not know whether vanishing of the odd trace implies vanishing of every odd-characteristic theta series in the same degree.

\begin{rem}
We expect that the shift in degree in
Proposition~\ref{odd-distinguish} can be omitted: a nonzero
odd trace should imply linear independence of the
odd-characteristic theta series in the same degree.
A separate question is whether vanishing of the odd trace
forces all its summands to vanish identically.
Neither assertion is established here.
\end{rem}

\section{Global theta-series spaces in ranks divisible by eight}
Fix $m\equiv0\pmod8$ and $g\ge1$. Let
\[
 \Theta=\Span_{\CC}\{\theta^{(g)}(M):M\text{ even unimodular of rank }m\}
\]
and
\[
 \Theta^{\pm}=\Span_{\CC}
 \{\theta_{L,g}^{\pm}:L\text{ odd unimodular of rank }m\}.
\]
The choice of characteristic representative does not affect these spans.
Let $c_0:\Theta\to\CC$ be the constant-coefficient functional.

\begin{prop}\label{q0}
Every pair of even unimodular $2$-neighbors is the pair of even neighbors
of an odd unimodular lattice.
\end{prop}
\begin{proof}
Write $C=M_1\cap M_2$ and $M_i=C\cup(C+x_i)$.
Since $2x_i\in C$, the pairing $(x_1,x_2)$ lies in $\tfrac12\ZZ$.
It cannot be integral: otherwise $M_1+M_2$ would be an integral lattice
properly containing the unimodular lattice $M_1$. Hence $x_1+x_2$ has
odd integral norm. The lattice
\[
 L=C\cup(C+x_1+x_2)
\]
is integral and has determinant one, since $\det C=4$ and $[L:C]=2$.
Its even sublattice is $C$, and its even neighbors are $M_1,M_2$.
\end{proof}

\begin{thm}[Kneser]\label{th}
For each fixed positive rank divisible by eight, the graph of isometry
classes of even unimodular lattices, with edges given by $2$-neighbors,
is connected.
\end{thm}
We use the classical connectivity theorem in this form; see \cite{Kn,Ch}.

\begin{thm}\label{global}
One has
\[
 \Theta^+=\Theta,\qquad \Theta^-=\ker(c_0|_\Theta).
\]
In particular, $\Theta^-$ has codimension one in $\Theta$.
\end{thm}
\begin{proof}
The trace identity gives $\Theta^+\subseteq\Theta$ and
$\Theta^-\subseteq\ker(c_0|_\Theta)$.
By Example~\ref{ex}, the two even neighbors of $\ZZ^m$ are isometric
to $D_m^+$. Hence $\theta_{\ZZ^m,g}^+=2^g\theta^{(g)}(D_m^+)$, so
$\theta^{(g)}(D_m^+)\in\Theta^+$.
For any even unimodular $M$, Theorem~\ref{th} gives a neighbor chain
\[
 D_m^+=M_0,M_1,\ldots,M_r\cong M.
\]
At the level of isometry classes, each edge can be represented by two
embedded neighbors. Proposition~\ref{q0} and the trace identity imply
$\theta^{(g)}(M_i)+\theta^{(g)}(M_{i+1})\in\Theta^+$.
Induction along the chain shows that $\theta^{(g)}(M)\in\Theta^+$.

For the second equality, write
$f=\sum_j a_j\theta^{(g)}(M_j)\in\ker c_0$.
Every lattice theta series has constant coefficient one, so $\sum_j a_j=0$.
Fix an even unimodular lattice $M_0$. Then
\[
 f=\sum_j a_j\bigl(\theta^{(g)}(M_j)-\theta^{(g)}(M_0)\bigr).
\]
Connect $M_j$ to $M_0$ by a neighbor chain. Each difference along an
edge belongs to $\Theta^-$, again by Proposition~\ref{q0} and
\eqref{trace}. Telescoping proves that every displayed difference lies
in $\Theta^-$, hence $f\in\Theta^-$.
Finally, $c_0$ is nonzero, which gives the codimension assertion.
\end{proof}

He  takes full responsibility for its
mathematical content.

\end{document}